\documentclass{amsart}

\usepackage{amsmath,amssymb}
\usepackage{amsthm}
\usepackage{graphicx}
\usepackage{hyperref}
\usepackage{xcolor}
\usepackage[margin=1.25in]{geometry}

\usepackage{graphicx,tikz}
\newtheorem{theorem}{Theorem}

\newcommand{\tr}{\mbox{tr}}

\newtheorem{lemma}{Lemma}

\theoremstyle{remark}

\theoremstyle{definition}

\newcommand{\ep}{\varepsilon}

\newcommand{\lam}{\lambda}

\newcommand{\N}{\mathbb{N}}

\newcommand{\beq}{\begin{equation}}
\newcommand{\eeq}{\end{equation}}

\begin{document}

\title[]{On Matrix Rearrangement Inequalities}

\author[]{Rima Alaifari, Xiuyuan Cheng, Lillian B. Pierce and Stefan Steinerberger}

\address[rima.alaifari@math.ethz.ch]{{Rima Alaifari}:  Department of Mathematics, ETH Z\"{u}rich, R\"{a}mistrasse 101, 8092 Z\"{u}rich}
\address[xiuyuan.cheng@yale.edu ..edu]{{Xiuyuan Cheng}: Department of Mathematics, Duke University, 120 Science Drive, Durham NC 27708}
\address[pierce@math.duke.edu]{{Lillian B. Pierce}: Department of Mathematics, Duke University, 120 Science Drive, Durham NC 27708}
\address[stefan.steinerberger@yale.edu]{{Stefan Steinerberger}: Department of Mathematics, Yale University, 10 Hillhouse Avenue, New Haven, 06511 CT}

\thanks{R.A. thanks David Gontier for fruitful discussions. X.C. is partially supported by the NSF (DMS-1818945,  DMS-1820827). L.P. is partially supported by CAREER grant NSF DMS-1652173 and the Alfred P. Sloan Foundation.  S.S. is partially supported by the NSF (DMS-1763179) and the Alfred P. Sloan Foundation.}

\begin{abstract} Given two symmetric and positive semidefinite square matrices $A, B$,   is
it true that any matrix given as the product of $m$ copies of $A$ and $n$ copies of $B$ in a particular sequence must be dominated in the spectral norm by the ordered matrix product $A^m B^n$? For example, is 
$$ \|  AABAABABB \| \leq  \|  AAAAABBBB \| ?$$
Drury \cite{drury} has characterized precisely which disordered words have the property that an inequality of this type holds for all matrices $A,B$. However, the $1$-parameter family of  counterexamples Drury constructs  for these characterizations is comprised of $3 \times 3$ matrices, and thus as stated the characterization applies only for $N \times N$ matrices with $N \geq 3$.
In contrast, we prove that for $2 \times 2$ matrices, the  general rearrangement inequality holds for all disordered words. We also show that for larger $N \times N$ matrices, the general rearrangement inequality holds for all disordered words, for most $A,B$ (in a sense of full measure) that are sufficiently small perturbations of the identity.
 \end{abstract}

\maketitle

\section{Introduction }
\subsection{Introduction.} 

Rearrangement inequalities for functions have a long history; we refer to Lieb and Loss \cite{lieb} for an introduction and an example of their ubiquity in Analysis, Mathematical Physics, and Partial Differential Equations. A natural question that one could ask is whether there is an operator-theoretic variant of such rearrangement inequalities. For example, given two operators $A:X \rightarrow X$ and $B:X \rightarrow X$, is there an inequality 
$$ \|  ABABA \| \leq  \|  AAABB \| $$
where $\| \cdot \|$ is a norm on operators? 
In this paper, we will study the question for $A,B$ being symmetric and positive semidefinite square matrices and $\| \cdot \|$ denoting the classical operator norm 
$$\| M \| =  \sup_{\|x\|_2=1} \|Mx \|_2.$$

We are interested in whether one could hope for a statement of the  general type
$$
	\|A^{m_1} B^{n_1} A^{m_2} B^{n_2} \cdots A^{m_s} B^{n_s} \| \leq \|A^m B^n\|,
$$
where  
\[ m=\sum_{j=1}^s m_j, \qquad n=\sum_{j=1}^s n_j ,\]
 with $m_j,n_j$  positive integers (except that we allow $m_1 = 0$ or $n_s=0$).
Of course, if the operators commute then any such inequality is trivially an equality.
A reason why one might hope in general for such a statement to be true is that one could expect the repeated application of only one operator to lead to growth (or at least preservation) of the norms of suitable eigenvectors, while alternating applications of two operators could have the effect of  projecting alternately onto two possibly different eigenbases, thus losing size of the eigenvectors.  

\subsection{Known results.} There are several encouraging results in this direction, some of which are by now classical in Operator Theory, and have been extended in a variety of different ways. We note:

\begin{itemize}
\item \textbf{Heinz-L\"owner inequality} (Heinz \cite{heinz}, 1951), (L\"owner \cite{lowner}, 1934) stating that
$$ \left\| A B A \right\| \leq  \left\| A A B \right\|.$$

\item \textbf{Heinz-Kato inequality} (Heinz \cite{heinz}, 1951), (Kato \cite{kato}, 1952). If $A,B$ are positive operators and $T$ is a linear operator such that $\| Tx \| \leq \|Ax\|$ and
$\|T^{*} y\| \leq \|By\|$ for all $x,y$ in a Hilbert space, then
$$ \left| \left\langle Tx, y\right\rangle \right|  \leq \left\| A^{\alpha} x \right\| \left\|B^{1-\alpha} y\right\|   \qquad \mbox{for all}~0 \leq \alpha \leq 1.$$

\item \textbf{Cordes inequality} (Cordes \cite{cord}, 1987). For all symmetric and positive definite $A,B$ and all $0 \leq s \leq 1$
$$ \left\| A^s B^s\right\| \leq \|AB\|^s.$$

\item \textbf{McIntosh's inequality} (McIntosh \cite{mc}, 1979) generalizes several of the earlier results and shows that for $A,B$ as above and $X$ an arbitrary square matrix of the same size,
$$\label{McIntosh_L2}
 \| AXB \| \leq \|A^2X\|^{1/2} \|X B^2\|^{1/2}.
$$
The last author characterized equality for several of these inequalities in \cite{stein}.

\item \textbf{Furuta's inequality}  \cite{furuta1} (see also \cite{cord}) shows that for any $n \geq 1$
$$\label{Furuta_orig2}
	\|AB\|^n \leq \|A^{n} B^{n}\|.
$$
\end{itemize}

There is a large literature connected to these inequalities; we refer to \cite{andr, corach1, dix, fu, ge, heinz2} as well as the books by Bhatia \cite{bhatia, bhatia2}, Cordes \cite{cord}, Furuta \cite{furuta2}, Marshall, Olkin \& Arnold \cite{marshall}, Simon \cite{simon} and Zhan \cite{zhan}. Many open problems remain. The authors themselves were motivated by a conjecture of Recht and R\'e \cite{recht} who asked whether, for $n$ positive definite matrices $A_1, \dots, A_n$, there is an inequality 
$$ \frac{1}{n^m} \left\| \sum_{j_1, \dots, j_m = 1}^{n}{A_{j_1} \dots A_{j_m} } \right\| \geq \frac{(n-m)!}{n!} \left\| \sum_{j_1, \dots, j_m = 1 \atop \mbox{\tiny all distinct} }{ A_{j_1} A_{j_2} \cdots A_{j_m}} \right\|.$$
Recht and R\'e \cite{recht} proved the inequality for $n=m=2$; Zhang \cite{zhang} recently gave a proof for $m=3$ and $n \geq 6$ being a multiple of 3. Israel, Krahmer and Ward \cite{israel} prove the inequality for $n=3$; we also refer to recent work of Albar, Junge and Zhao \cite{albar}. One way of interpreting the conjectured inequality of Recht and R\'e is that repetition of matrices has a beneficial effect on the operator norm; this leads  to asking about matrix rearrangement inequalities, as studied in this paper.

\subsection{Statement of results.} Consider a putative inequality
\beq\label{gen_ineq}
\| A^{m_1} B^{n_1} A^{m_2} B^{n_2} \cdots A^{m_s} B^{n_s}\| \leq \|A^{m} B^{n}\|,
\eeq
where $m=\sum_i m_i, n=\sum_i n_i$ and $m_i, n_i \in \mathbb{N}$ (possibly allowing $m_1=0$ or $n_s=0$) and $A,B$ are symmetric and positive semidefinite square matrices. Is it true that given any ``word,'' that is, a tuple of exponents $(m_1,n_1,\ldots,m_s,n_s)$, the inequality (\ref{gen_ineq}) holds for all such $A,B$?
 Drury \cite{drury} has shown that at this level of generality, the question has a negative answer. Moreover, he provides a complete characterization of conditions on the exponents $(m_1,n_1,\ldots,m_s,n_s)$ for which such an inequality holds for all such $A,B$ (of all dimensions).  For example, Drury shows that we always have
$$ \| AABBABBAABBAA\| \leq \|A^7 B^6\|$$
while the inequality
$$\|AABABB\| \leq \|A^3B^3\|   \quad \mbox{can fail for certain}~ A,B.$$
The counterexamples given by Drury to the general rearrangement inequality   stem from a 1-parameter family of $3 \times 3$ matrices. 
In contrast, our first main result is that the general rearrangement inequality does indeed hold true \emph{for any word, for all $2 \times 2$ symmetric positive semidefinite matrices.}
\begin{theorem}[General Rearrangement Inequality for $2\times 2$ Matrices]
Let $A, B$ be symmetric positive semidefinite matrices of size $2 \times 2$ and let $m_i, n_i \in \mathbb{N}$ (possibly allowing $m_1=0$ or $n_s=0$). Then
$$ \| A^{m_1} B^{n_1} A^{m_2} B^{n_2} \cdots A^{m_s} B^{n_s}\| \leq \|A^{m} B^{n}\|,$$
where $m=\sum_i m_i$ and $n= \sum_i n_i$.
\end{theorem}

In light of Drury's results, there is no hope for such general inequalities in higher dimensions. Nonetheless, one could wonder whether there is hope that, given any word $(m_1,n_1,\ldots, m_s,n_s)$, a rearrangement inequality should hold for \textit{some} (or maybe even \textit{most}) pairs of $N\times N$ matrices $(A,B)$.
This is the motivation for our second result, which states that given any word, the rearrangement inequality is generically true for $N \times N$ matrices in a sufficiently small neighborhood of the identity, for all $N \geq 2$. 
\begin{theorem}[General Rearrangement close to the Identity, arbitrary dimension]
Let $A, B$ be symmetric positive semidefinite matrices and let $m_i, n_i \in \mathbb{N}$ (possibly allowing $m_1=0$ or $n_s=0$). If $\mbox{ker}(AB - BA) = \emptyset$, then there exists $\varepsilon_0 = \ep_0(A,B,m,n) > 0$  such that for all $0 < \varepsilon < \varepsilon_0$
$$ \| (\emph{Id} + \varepsilon A)^{m_1}  (\emph{Id} + \varepsilon B)^{n_1} \cdots  (\emph{Id} + \varepsilon A)^{m_s}  (\emph{Id} + \varepsilon B)^{n_s}\| \leq \| (\emph{Id} + \varepsilon A)^{m}  (\emph{Id} + \varepsilon B)^{n}\|,
$$
where $m=\sum_i m_i$ and $n=\sum_i n_i$.
\end{theorem}

Thus given any fixed word, this provides a codimension 1 family of $(A,B)$ among all relevant pairs of $N \times N$ matrices in the neighborhood of the identity, which satisfy the   rearrangement inequality for that word.
We do not know whether the condition $\mbox{ker}(AB - BA) = \emptyset$ is necessary but are inclined to think that it may not be. 

There are many other natural questions that come to mind. The rearrangement inequalities are invariant under multiplication with constants, which allows us to compactify the set of matrices: are such inequalities \textit{generically} true (in, say, the sense that the measure of admissible matrices approaches full measure as the length of the inequality, or the number $s$, increases)?  Another question could be to determine other simple conditions on the matrices (other than assuming that they commute) that would imply the desired rearrangement inequalities hold.

\section{Proof of Theorem 1}
Our proof uses three different ingredients. 
The first ingredient is
Corollary 4.4 in a paper of Ando, Hiai \& Okubo \cite{ando} which states the following: let $C,D$ be symmetric positive semidefinite matrices of size $2\times2$ and for $i=1, \dots, k$, let $p_i, q_i \geq 0$ satisfy
 $$p_1 + \cdots +p_k = 1 = q_1 + \cdots + q_k;$$ 
then 
\begin{equation}\label{ineq_ando}
\tr(C^{p_1} D^{q_1} \cdots C^{p_k} D^{q_k}) \leq \tr(CD).
\end{equation}
We remark that Ando, Hiai \& Okubo \cite{ando} were motivated by the question whether such a trace inequality might be true in general: they establish the result for general positive semidefinite matrices that have at most two distinct eigenvalues. Plevnik \cite{plevnik} recently constructed an example showing that (\ref{ineq_ando}) can fail for $3 \times 3$ matrices.

The second ingredient is the invariance of trace with respect to cyclic permutations, i.e.
\beq\label{cyclic}
 \tr(A_1 A_2 \cdots A_n) = \tr(A_2 A_3 \cdots A_{n-1} A_n A_1).
\eeq
The third ingredient is the basic equation
\beq\label{max}
 \|A\|^2 = \max_{\|x\|=1}{\left\langle Ax, Ax \right\rangle} = \max_{\|x\|=1}{\left\langle A^T Ax, x \right\rangle} = \lambda_{\max}(A^T A),
 \eeq
where $\lambda_{\max}$ denotes the largest eigenvalue of a matrix.

Let $A,B$ be symmetric positive semidefinite matrices of size $2 \times 2$. 
Consider now a general word 
\beq\label{word}
W_{m,n}(A,B) := A^{m_1} B^{n_1} \cdots A^{m_s} B^{n_s}
\eeq
 where 
$$m = m_1 + \cdots + m_s \quad \mbox{and} \quad n = n_1 + \cdots + n_s.$$
Assume the symmetric matrices $B^m A^{2n} B^m$ and  $W_{m,n}(A,B)^T W_{m,n}(A,B)$ have eigenvalues (not necessarily distinct) given by
$$   \sigma(B^n A^{2m} B^n) = \left\{\lambda_1, \lambda_2\right\} \qquad \mbox{and} \qquad \sigma( W_{m,n}(A,B)^T W_{m,n}(A,B))=\left\{ \mu_1, \mu_2\right\}.$$
We note that all these eigenvalues are nonnegative. 
Moreover, assuming the ordering $\lambda_1 \geq \lambda_2$ and $\mu_1 \geq \mu_2$, we have by (\ref{max}) that
$$ \lambda_1 = \|A^m B^n\| ^2 \qquad \mbox{and} \qquad \mu_1 = \| W_{m,n}(A,B)\|^2.$$
Thus to prove Theorem 1, it suffices to show that 
\beq\label{thm1_max}
 \mu_1 \leq \lambda_1.
\eeq
 Defining $C= A^{2m}, D=B^{2n}$, we employ the cyclic identity (\ref{cyclic}) followed by (\ref{ineq_ando}) with $k=2s$ and
 \[ (p_1,q_1,\ldots, p_s,q_s) = \left(\frac{m_s}{2m}, \frac{n_{s-1}}{2n}, \cdots, \frac{n_1}{2n}, \frac{2m_1}{2m}, \frac{n_1}{2n}, \ldots, \frac{m_s}{2m}, \frac{2n_s}{2n}\right),\]
followed by a second application of the cyclic identity (\ref{cyclic})
 to obtain
\begin{align*}
\tr(W_{m,n}(A,B)^T W_{m,n}(A,B)) &= \tr(A^{m_s} B^{n_{s-1}} \cdots B^{n_1} A^{2 m_1} B^{n_1} \cdots A^{m_s} B^{2 n_s}) \\
&\leq \tr(C D) = \tr(B^n A^{2m} B^n).
\end{align*}
Since the trace is merely the sum of the eigenvalues, this shows that
\beq\label{trace_add}
 \mu_1 + \mu_2 \leq \lambda_1 + \lambda_2.
 \eeq
 On the other hand, the determinant is multiplicative, and so
\beq\label{det}
 \lambda_1 \cdot \lambda_2 = \det(B^n A^{2m} B^n) = \det(W_{m,n}(A,B)^T W_{m,n}(A,B)) = \mu_1 \cdot \mu_2.
 \eeq
 It is simple to deduce from these two relations that (\ref{thm1_max}) must hold.
 
 Indeed, if $\mu_1 = 0$, we have the desired result (\ref{thm1_max}). If $\mu_1 \neq 0$ but $\mu_2 = 0$ then either $\lambda_1$ or $\lambda_2$ must vanish, by (\ref{det}). If $\lambda_1=0$, then $\lambda_1 \geq \lambda_2$ implies that $\lambda_2=0$ and we have a contradiction to (\ref{trace_add}). Thus in this case we must have $\lambda_2= 0$, and then the desired inequality (\ref{thm1_max}) follows from (\ref{trace_add}). It remains to deal with the case when $\lambda_1$ and $\lambda_2$ are both nonzero, which implies that $\mu_1 \geq \mu_2 > 0$.
Suppose contrary to (\ref{thm1_max}) that $\mu_1 = \lambda_1 + \delta_1$ for some $\delta_1 > 0$; then (\ref{trace_add}) implies that  $\mu_2 = \lambda_2 - \delta_2$ for some $\delta_2 \geq \delta_1 > 0$. Then by (\ref{det}),
\begin{align*}
\lambda_1 \lambda_2 &= \mu_1 \mu_2 = (\lambda_1 + \delta_1)(\lambda_2 - \delta_2)\\
&= \lambda_1 \lambda_2 +\delta_1 \lambda_2 - \lambda_1 \delta_2 - \delta_1 \delta_2\\
&\leq \lambda_1 \lambda_2 - \delta_1\delta_2
\end{align*}
which is the desired contradiction. (Alternatively one can use (\ref{trace_add}) and (\ref{det}) to prove, using induction and repeated squaring of both sides of (\ref{trace_add}), that for any $k \in \mathbb{N}$, $mu_1^{2^k} + \mu_2^{2^k} \leq \lambda_1^{2^k} + \lambda_2^{2^k}.$
For such  expressions the leading term is asymptotically dominant and this shows $\mu_1 \leq \lambda_1$.)
 This verifies (\ref{thm1_max}) and hence completes the proof of Theorem 1.

\section{Proof of Theorem 2}\label{sec_thm_2}
Let $A,B$ be fixed symmetric positive semidefinite $N \times N$ matrices, and assume that the tuple of exponents $(m_1,n_1,\ldots, m_{s},n_s)$ is fixed, with $m=\sum m_i$ and $n = \sum n_i$. Let $W_{m,n}(\mbox{Id}+\ep A,\mbox{Id}+\ep B)$ denote the corresponding word in terms of $\mbox{Id}+\ep A,\mbox{Id}+\ep B$, analogous to (\ref{word}). 
The proof idea can be summarized as follows. Let $X_\ep$ denote $W_{m,n}(1+\ep A, 1+\ep B)$, and let $Z_\ep$ denote $(\mbox{Id}+\ep A)^m (\mbox{Id}+\ep B)^n$. We will choose a vector $v_\ep$ with $\|v_\ep \|=1$ that maximizes 
\[\| X_\ep v_\ep \|^2 = \langle X_\ep  v_\ep , X_\ep v_\ep \rangle.\]
Then as long as we can show that for this $v_\ep$ we have 
\beq\label{main_goal}
\|X_\ep v_\ep \|^2 \leq \| Z_\ep v_\ep \|^2,
\eeq
we can conclude that 
\beq\label{main_goal_app}
 \|X_\ep\| =  \| X_\ep v_\ep\| \leq \| Z_\ep v_\ep \| \leq \sup_{\|v\|=1} \| Z_\ep v\| = \| Z_\ep\|,
\eeq
thus proving Theorem 2. 

By simply multiplying out $\| X_\ep v_\ep \|^2$ and $\| Z_\ep v_\ep \|^2$, we will see that the leading order terms (in $\ep$)  come in both cases from a matrix of the form
$$ \mbox{Id} + \varepsilon(m A + n B) + \mathcal{O}(\varepsilon^2).$$
This motivates us to show that a significant proportion of $v_\ep$ must lie in the eigenspace of the largest eigenvalue of the matrix $mA + nB$ (Lemma 1 below). This observation will suffice to examine terms up to second order in $\ep$ in the desired inequality (\ref{main_goal}). Next, to treat the terms of third order and higher in $\ep$, we will use a second lemma (Lemma 2 below), which shows that if $\ker(AB-BA)=\emptyset$, for an eigenvector corresponding to the largest eigenvalue of $mA+nB$, the third order terms provide a strict inequality. This therefore allows us to neglect all higher order terms in $\ep$ (as long as $\ep$ is sufficiently small), and that leads to the desired inequality (\ref{main_goal}). 

%

\subsection{Two Lemmata}
Our first lemma states that a one-parameter family of matrices that
is approximately given by the identity plus a small linear term $\ep Y$ has the property that the eigenvector corresponding to its largest eigenvalue is necessarily very close to the leading
eigenspace of the linear perturbation $Y$. This statement is certainly not novel, but we provide its simple proof.

\begin{lemma} Let $X_{\varepsilon} = \emph{Id} + \varepsilon Y+ \mathcal{O}(\varepsilon^2)$, where $Y$ is a symmetric positive semidefinite matrix and $\ep$ varies, giving a one-parameter family. For each $\ep>0$ let $v_{\varepsilon}$ be a
vector satisfying $\| v_{\varepsilon}\| =1$ and
$$ \|X_{\varepsilon} v_{\varepsilon}\| = \|X_{\varepsilon}\|.$$
Let $\pi_{}$ be the orthogonal projection onto the eigenspace of the largest eigenvalue of $Y$.
Then there exists a constant $C_1 = C_1(Y)>0$ and  also $\ep_0 = \ep_0(Y)>0$ such that for every $0<\ep< \ep_0$,  
 $$ \| \pi_{} v_{\varepsilon}  \| \geq 1 - C_1 \varepsilon.$$ 
\end{lemma}
\begin{proof} 
Let us simplify notation and write $v = \pi_{} v_{\varepsilon}$ and $w = v_{\varepsilon} - v$. Observe that they are orthogonal and thus
$$\|v\|^2 + \|w\|^2 = 1.$$
We have, expanding up to first order,
\begin{align*}
  \|X_{\varepsilon} v_{\varepsilon}\|^2 = 1 + 2 \varepsilon \left( \left\langle Y v, v\right\rangle + \left\langle Y v, w\right\rangle + \left\langle Y w, v\right\rangle + \left\langle Y w, w\right\rangle\right) + \mathcal{O}(\varepsilon^2),
  \end{align*}
  in which the implicit constant depends on $Y$.
We will now see that several terms simplify. If $Y$ has only one eigenvalue, then the projection $\pi$ is merely the identity and the result follows.
From now on we may suppose that $Y$ has at least two distinct eigenvalues and we use $\lambda_1$ to denote the largest eigenvalue of $Y$ and $\lambda_2 < \lambda_1$ to denote the next largest. 
Then
\begin{align*}
  \left\langle Y v, v\right\rangle &= \lambda_1 \|v\|^2,  \qquad \left\langle Y v, w\right\rangle  =  \left\langle \lambda_1 v, w\right\rangle  = 0\\
\left\langle Yw, v\right\rangle  &=\left\langle w,  Y v\right\rangle = 0, \qquad \left\langle Yw , w\right\rangle \leq \lambda_2 \|w\|^2. 
\end{align*}
Altogether we have
\beq\label{upper}
  \|X_{\varepsilon} v_{\varepsilon}\|^2 \leq 1 + 2\varepsilon(\lambda_1 \|v\|^2 + \lambda_2 \|w\|^2 )+ \mathcal{O}(\varepsilon^2).
  \eeq

We recall that $v_\ep$ was chosen to maximize $  \|X_{\varepsilon} u\|^2 $ over all $\|u\|=1$. In  particular, if $u$ is an eigenvector of $Y$ for $\lam_1$ with $\|u \|=1$, then 
\[  \|X_{\varepsilon} v_{\varepsilon}\|^2  \geq \|X_{\varepsilon} u\|^2  = \| (\mathrm{Id}+\ep Y +  \mathcal{O}(\ep^2))u\|^2 = 1 + 2\ep \lam_1 + \mathcal{O}(\ep^2).\]
Applying this in (\ref{upper}) shows that there is a constant $c>0$ (depending on the implicit constants in the $\mathcal{O}(\ep^2)$ terms, and hence on $Y$) such that as long as $\ep$ is sufficiently small (again relative to the implicit constants in the $\mathcal{O}(\ep^2)$ terms),
$$ \lambda_1 \|v\|^2 + \lambda_2 \|w\|^2  \geq \lambda_1 - c \varepsilon.$$
Using $\|w\|^2 = 1-\|v\|^2$, we obtain
$$ \|v\|^2 \geq 1 - c' \varepsilon$$
where $c' = c/(\lambda_1 - \lambda_2)$ and therefore
$$ \|v\| \geq 1 - C_1 \varepsilon,$$
where $C_1 = c'(1-\ep/4) \geq c'/2$ for all $\ep < \ep_0(Y)$ sufficiently small, for a parameter $\ep_0(Y)$ depending only on $c, \lam_1,\lam_2$, and hence only on $Y$. 
\end{proof}

Our second lemma states rearrangement inequalities for an eigenvector of the largest eigenvalue of $mA+nB$ (motivated by Lemma 1). The argument
is again elementary but the statement itself is so specific that it is presumably new.

\begin{lemma}Let $A,B$ be symmetric and positive semidefinite square matrices such that $\ker(AB - BA) = \emptyset$. Fix $m,n \in \mathbb{N}$ and let $\lam_1$ denote the largest eigenvalue of $mA + nB$. Then  there exists a constant $C_2 = C_2(n,m,A,B)>0$ such that for all vectors $\|v\|=1$ satisfying 
\beq\label{v_prop}
(mA + nB)v = \lambda_1 v
\eeq
we have the inequalities
$$ \left\langle ABA v, v \right\rangle \leq  \left\langle AAB v, v \right\rangle - C_2 \|v\|^2,$$
$$ \left\langle BABv, v \right\rangle \leq \left\langle ABBv, v \right\rangle - C_2 \|v\|^2.$$
\end{lemma}
\begin{proof} We start by showing the first inequality. We claim
$$ B \leq \frac{\lambda_1 \mbox{Id} - mA}{n}$$
in the sense that
$$ \frac{\lambda_1 \mbox{Id} - mA}{n}  - B$$
is positive semidefinite. Indeed, we have that 
\beq\label{diamond}
  \left\langle \left(\frac{\lambda_1 \mbox{Id} - mA}{n}  - B\right)x, x\right\rangle \geq 0 
\eeq
with equality if and only if $x$ is an eigenvector of $mA + nB$ corresponding to eigenvalue $\lambda_1$.
This holds since $mA + nB$ is symmetric and positive semidefinite and its operator norm thus coincides with its largest eigenvalue. 
We now suppose $v$ with $\|v\|=1$ satisfies (\ref{v_prop}). 
Solving for $Bv$ in $(mA+nB)v = \lambda_1v$, we can rewrite
\beq
 \left\langle AAB v, v \right\rangle =  \left\langle AB v, Av \right\rangle = \left\langle A \left( \frac{\lambda_1 \mbox{Id}  - mA}{n}\right) v, Av \right\rangle = \left\langle \left( \frac{\lambda_1 \mbox{Id}  - mA}{n}\right)A v, Av \right\rangle. \label{diamond2}
\eeq
We now need to compare this to $ \left\langle ABA v, v \right\rangle$ which we can rewrite as
$$ \left\langle ABA v, v \right\rangle  =  \left\langle BA v,Av \right\rangle;$$
subtracting this from (\ref{diamond2}) we see by (\ref{diamond}) that 
\beq\label{diamond3}
   \left\langle AAB v, v \right\rangle-\left\langle ABA v, v \right\rangle= \left\langle \left(\frac{\lambda_1 \mbox{Id} - mA}{n}  - B\right)Av, Av\right\rangle \geq 0 .
   \eeq
Now we aim to show that this inequality is strict if $v$ satisfies (\ref{v_prop}). From our previous observation about (\ref{diamond}), we know that equality holds in this last inequality precisely when $Av$  is an eigenvector of $mA + nB$ corresponding to eigenvalue $\lambda_1$. Suppose this is true. Then
$$ (mA + nB)(Av) = \lambda_1 Av$$
while on the other hand, multiplying our assumption (\ref{v_prop}) by $A$  on the left-hand side shows that
$$ A (mA + nB)v = A \lambda_1 v.$$ 
Subtracting these two identities shows that  $AB v = BAv$, violating our assumption $\ker(AB-BA) = \emptyset$. 
We conclude that $Av$ cannot be an eigenvector for $mA+nB$ corresponding to $\lam_1$, and hence the inequality in (\ref{diamond3}) is strict, for any $\|v\|=1$ satisfying (\ref{v_prop}). By compactness of
the unit ball $\left\{v : \|v\| =1\right\}$, there exists a constant $C_2=C_2(n,m,A,B)>0$ such that
$$ \left\langle ABA v, v \right\rangle \leq  \left\langle AAB v, v \right\rangle - C_2 \|v\|^2,$$
concluding the proof of the first claim.
As for the second inequality, we relabel $A$ and $B$, obtain from the first case that
$$ \left\langle BAB v, v \right\rangle \leq  \left\langle BBA v, v \right\rangle - C_2 \|v\|^2,$$
and note that $\left\langle BBA v, v \right\rangle = \left\langle A v, BBv \right\rangle = \left\langle ABB v, v \right\rangle$.

\end{proof}

\subsection{Conclusion of the proof of  Theorem 2}

We are now ready to prove Theorem 2. We recall from the beginning of \S \ref{sec_thm_2} that we consider a particular word $X_\ep=W_{m,n}(1+\ep A,1+\ep B)$, with sequences of exponents $(m_1,n_1,\ldots, m_s,n_s)$ and $m = \sum_i m_i$, $n = \sum_i n_i$. We let $Z_\ep$ denote $(1+\ep A)^m (1+\ep B)^n$. We  choose a vector $v_\ep$ with $\|v_\ep \|=1$ that maximizes 
$\| X_\ep v_\ep \|^2 = \langle X_\ep  v_\ep , X_\ep v_\ep \rangle,$
and it suffices to show  that $\|X_\ep v_\ep \|^2 \leq \| Z_\ep v_\ep \|^2$,  as explained in (\ref{main_goal_app}).
We will expand the unordered product $\| X_\ep v_\ep \|^2 $ and the ordered product $\| Z_\ep v_\ep \|^2 $ up to the third term, with respect to $\ep$. Lemma 1 will restrict the types of vectors we will have to study, Lemma 2 will give us a strict inequality in the third order terms, and the desired inequality will follow from that.

Precisely, in the above setting we will prove that there exist positive constants $C_1,C_2,\ep_0$ depending on $A,B,m,n$ such that for all $\ep< \ep_0$,
\beq\label{key_ineq}
 \| Z_\ep v_\ep \|^2  -  \| X_\ep v_\ep \|^2  \geq \varepsilon^3 C_2(1 - C_1\ep)^2 + \mathcal{O}(\ep^4).
\eeq
Here the implicit constant depends on $A,B,m,n$.
Consequently, for all sufficiently small $\ep$, the left-hand side is in fact strictly positive, and Theorem 2 follows.

A simple expansion shows that
$$ X_{\varepsilon} = \mbox{Id} + \varepsilon X_1+ \ep^2 X_2 + \ep^3 X_3 + \mathcal{O}(\varepsilon^4)$$
$$ Z_{\varepsilon} =  \mbox{Id} + \varepsilon Z_1 + \ep^2 Z_2 + \ep^3 Z_3 + \mathcal{O}(\varepsilon^4),$$
where 
\[ X_1 = mA + nB = Z_1\]
and, for combinatorial coefficients $a_i \in \N$ depending only on the sequences of exponents $m_i$ and $n_i$,
\begin{align*}
 X_2 &= a_1 AA + a_2 AB + a_3 BA + a_4 BB, \\
 Z_2 &= a_1 AA + (a_2+a_3) AB + a_4 BB,
\end{align*}
and 
\begin{align*}
 X_3 &= a_5 AAA + a_6 AAB + a_7 ABA + a_8 BAA + a_9 ABB + a_{10} BAB +a_{11} BBA + a_{12} BBB, \\
 Z_3 &= a_5 AAA + (a_6+a_7 + a_8) AAB + (a_9+a_{10}+a_{11})ABB + a_{12}BBB.
\end{align*}
Thus an expansion up to third order shows that for any $u$,
\begin{align}
\| X_\ep u\|^2 = \left\langle X_\ep u, X_\ep u \right\rangle &= \|u\|^2 + 2\varepsilon \left\langle X_1 u, u \right\rangle \nonumber \\
&+ 2\varepsilon^2 \left\langle X_2 u, u\right\rangle + \varepsilon^2 \left\langle  X_1 u, X_1 u \right\rangle  \nonumber \\
&+ 2 \varepsilon^3 \left\langle X_3 u, u \right\rangle + 2\varepsilon^3 \left\langle X_2 u, X_1 u \right\rangle + \mathcal{O}(\varepsilon^4) \label{X_exp}
\end{align}
and
\begin{align}
\| Z_\ep u\|^2 = \left\langle Z_\ep u, Z_\ep u \right\rangle &= \|u\|^2 + 2\varepsilon \left\langle Z_1 u, u \right\rangle  \nonumber\\
&+ 2\varepsilon^2 \left\langle Z_2 u, u\right\rangle + \varepsilon^2 \left\langle  Z_1 u, Z_1 u \right\rangle  \nonumber \\
&+ 2 \varepsilon^3 \left\langle Z_3 u, u \right\rangle + 2\varepsilon^3 \left\langle Z_2 u, Z_1 u \right\rangle + \mathcal{O}(\varepsilon^4). \label{Z_exp}
\end{align}

We now use $v_{\varepsilon}$ to denote the vector maximizing $\| X_{\varepsilon} u \|$ among all $\|u \|=1$,
and we aim to show the inequality (\ref{key_ineq}) for 
\beq\label{key_diff}  \left\langle Z_\ep v_\ep, Z_\ep v_\ep \right\rangle - \left\langle X_\ep v_\ep, X_\ep v_\ep \right\rangle,
\eeq
 using the above expansions (\ref{X_exp}) and (\ref{Z_exp}).

We will first see that terms in this difference that are at most second order in $\ep$ cancel exactly, in fact for any $u$. 
Indeed, the term of order 0 in $\varepsilon$, that is $\|u\|^2$, cancels and, since $X_1 = Z_1$, so does the term of order $\varepsilon$. 
Next, for the second order terms, for any vector $u$,
\begin{align*}
 \left\langle Z_2 u, u\right\rangle -  \left\langle X_2 u, u\right\rangle &= a_3 \left\langle ABu, u\right\rangle - a_3 \left\langle BA u, u \right\rangle  \\
&=  a_3 \left\langle Bu, Au\right\rangle -a_3 \left\langle A u, Bu \right\rangle= 0.
\end{align*}
The other terms of second order,  $\left\langle  Z_1 u, Z_1 u \right\rangle$ and $\left\langle  X_1 u, X_1 u \right\rangle$ again coincide trivially (and hence cancel in the difference) because $X_1 = mA + nB = Z_1$.
This shows that for any $u$, the terms of at most second order (with respect to $\ep$) cancel in  the difference (\ref{key_diff}). \\ 

We now analyze the third order terms in the difference (\ref{key_diff}), which include terms of two types, namely
\begin{align*}
 \left\langle Z_3 u,u\rangle -\langle X_3 u, u\right\rangle &= a_7 (\left\langle AABu, u\right\rangle - \left\langle ABAu, u\right\rangle ) + a_8 (\left\langle AABu, u\right\rangle - \left\langle BAAu, u\right\rangle ) \nonumber \\
	& \qquad+ a_{10} (\left\langle ABBu, u\right\rangle - \left\langle BABu, u\right\rangle ) + a_{11} (\left\langle ABBu, u\right\rangle - \left\langle BBAu, u\right\rangle ) 
\end{align*}
and
\beq\label{type2}
 \left\langle Z_2 u,Z_1 u\rangle -\langle X_2 u, X_1 u\right\rangle = a_3 \left\langle (AB - BA)u, (mA + nB)u \right\rangle .
\eeq
For the first type of term, we can use the fact that $\langle AAB u,u \rangle = \langle ABu,Au \rangle = \langle Au,ABu\rangle = \langle BAAu,u\rangle$
to see the terms corresponding to $a_8$ vanish, and similarly for $a_{11}$, so that
\beq\label{type1}
 \left\langle Z_3 u,u\rangle -\langle X_3 u, u\right\rangle =  a_7 (\left\langle AABu, u\right\rangle - \left\langle ABAu, u\right\rangle ) + a_{10} (\left\langle ABBu, u\right\rangle - \left\langle BABu, u\right\rangle ). 
 \eeq
Altogether, we obtain that for any $u$, the third order contributions of the difference (\ref{key_diff}) are given by
\begin{align}
\mbox{third order} &= 2 \varepsilon^3 a_7  \left( \left\langle AABu, u\right\rangle - \left\langle ABAu,u\right\rangle\right) + 2 \varepsilon^3 a_{10}  (\left\langle ABBu, u\right\rangle - \left\langle BABu, u\right\rangle ) \nonumber \\
& \qquad + 2\varepsilon^3 a_3 \left\langle (AB - BA)u, (mA + nB)u \right\rangle. \label{types}
\end{align}

Now we specialize to considering $u=v_\ep$ with $\|v_\ep\|=1$ that maximizes $\|X_\ep v_\ep \|^2$. We apply Lemma 1 to conclude that there is an $\ep_0 = \ep_0(mA+nB)>0$ and a constant $C_1 = C_1(mA+nB)>0$ such that for every $\ep< \ep_0 $, we can write 
$$v_\ep = v + w,$$ 
where $v$ is the projection of $v_\ep$ onto the eigenspace corresponding to the largest eigenvalue of $mA + nB$, and $v,w$ have the following properties: $\|v\| \geq 1 - C_1\ep$ and $v$ is orthogonal to $w$, so that $\|w\| \leq C_1 \ep$. (We note that both $v$ and $w$ also depend on $\varepsilon$ but suppress this for simplicity of notation). 
In (\ref{type2}) we see that 
\begin{align*}
 \left\langle (AB - BA)v_{\varepsilon}, (mA + nB)(v+w) \right\rangle &= \lambda_1 \left\langle (AB - BA)v_{\varepsilon},v\right\rangle + \left\langle (AB - BA)v_{\varepsilon}, (mA+nB)w\right\rangle \\
&=  \lambda_1 \left\langle (AB - BA)v_{\varepsilon},v\right\rangle + \mathcal{O}(\varepsilon)\\
&= \lambda_1 \left\langle (AB - BA)v,v\right\rangle + \mathcal{O}(\varepsilon)\\
&= \mathcal{O}(\varepsilon),
\end{align*}
since the first term vanishes; thus this type of term contributes 
$$ 2\varepsilon^3 a_3 \left\langle (AB - BA)v_\ep, (mA + nB)v_\ep \right\rangle = \mathcal{O}(\varepsilon^4)$$
to (\ref{types}).
A similar expansion for the other terms (\ref{type1}) shows that
$$  \left\langle Z_3 v_\ep, v_\ep \rangle - \langle X_3 v_\ep, v_\ep\right\rangle  =  \langle (Z_3-X_3) v,v \rangle  + \mathcal{O}(\ep)$$
with an implicit constant depending on $A,B,m,n$.
Now that we have restricted to an inner product involving only $v$,  we apply (\ref{type1}) for the vector $v$ and note that Lemma 2 implies that 
\[  \left\langle (Z_3 - X_3 )v, v\right\rangle \geq C_2(a_7+a_{10}) \|v\|^2 \geq C_2 \|v\|^2 \geq C_2(1 - C_1\ep)^2, \]
with the constant $C_2$ provided by the lemma. This is \emph{strictly positive} for all $\ep$ sufficiently small with respect to $C_1$. 
To conclude, we have proved (\ref{key_ineq}), and 
this completes the proof of Theorem 2.

%

\end{document}